\documentclass[11pt]{article}

\usepackage{sectsty}
\usepackage{titlesec}
\allsectionsfont{\centering \normalfont\scshape \large}
\titlelabel{\thetitle.\quad}

\usepackage{enumerate}
\usepackage{amssymb}
\usepackage{amsthm}
\usepackage{amsmath}
\usepackage{float}

\newcommand{\norm}[1]{\left\lVert#1\right\rVert}

\def\maxbool{\mathrel{%
    \mathchoice{\QEQ}{\QEQ}{\scriptsize\QEQ}{\tiny\QEQ}%
}}

\def\QEQ{{%
    \setbox0\hbox{$\cup$}%
    \rlap{\hbox to \wd0{\hss $\text{ }\vee$ \hss}}\box0
}}

\def\minbool{\mathrel{%
    \mathchoice{\REQ}{\REQ}{\scriptsize\REQ}{\tiny\REQ}%
}}

\def\REQ{{%
\setbox0\hbox{$\cup$}%
    \rlap{\hbox to \wd0{\hss $\text{ }\wedge$ \hss}}\box0
}}

\author{
  Jorge Garza Vargas\\ UC Berkeley \\
  E-mail: jgarzavargas@berkeley.edu 
  \and
 Dan Virgil Voiculescu \footnote{Research supported in part by NSF grant DMS-166534.} \\ UC Berkeley \\
  E-mail: dvv@math.berkeley.edu 
}

\date{}
\title{\Large{\textbf{Boolean Extremes and Dagum Distributions}}}

\begin{document}
\maketitle

\newtheorem{theorem}{Theorem}[section]
\newtheorem{definition}{Definition}[section]
\newtheorem{lemma}{Lemma}[section]
\newtheorem{corollary}{Corollary}[section]
\newtheorem{proposition}{Proposition}[section]

\begin{abstract}
We study the max-convolution and max-stable laws for Boolean independence and prove that these are Dagum distributions (also known as log-logistical distributions).
\end{abstract}

\section{Introduction}

\let\thefootnote\relax\footnotetext{2010 Mathematics Subject Classification 46L53, 60G70.} \footnote{Key words: Boolean independence, log-logistic distributions, extreme values, max stable laws.} In extreme value theory, instead of the sum of independent random variables, one deals with their max.  In \cite{Arous-Voiculescu}, looking for an analogue of extreme value theory in free probability, the Ando max was used. The Ando max \cite{Ando} of two noncommutative random variables, i.e. hermitian operators, amounts to replacing the operators by the projection-valued process, indexed by $\mathbb{R}$, where to each $t\in \mathbb{R}$ we associate the spectral projection for $(-\infty, t]$ and take the min of the projections. 

In this note we want to pursue the same idea with Boolean independence \cite{Speicher}, \cite{Speicher-Woroudi}. In the free setting the free max-stable laws turned out to fall into 3 families \cite{Arous-Voiculescu} as in the classical framework. Actually, in the free case, the max-stable laws resulted in generalized Pareto laws, coinciding surprisingly with precisely those limit laws which were found in Peaks over Thresholds in \cite{Balkema} (concerning this coincidence see \cite{Grela-Nowak}). Due to special features of Boolean independence related to units, it turns out that in this framework the question should be naturally restricted to positive noncommutative random variables. Instead of three families of limit laws there is only one family of Boolean max-stable laws and again, like with free independence, these laws have also been found in the classical framework, they are Dagum laws (see \cite{Kleiber}), which appear typically in wealth distributions. 

In addition to this introduction, this paper has three more sections. In section 2 we give preliminaries about Boolean independence (see for instance \cite{Bercovici}, \cite{Chackraborty-Hazra}, \cite{Speicher}, \cite{Speicher-Woroudi}), our choices between equivalent approaches, the Ando max and a discussion about the reasons for restricting the study of Boolean max-convolution to non-negative variables (or some equivalent class). 

In section 3 we compute the distribution of the Ando max of Boolean independent projections and from this result develop the Boolean max-convolution. We show that the semigroup on $[0, 1]$ with the operation coming from operations on Boolean independent projections is isomorphic to $[0,1]$ endowed with multiplication, which is the semigroup underlying classical extreme value theory. This isomorphism will be of great help and should be compared to the fact that in the case of free extreme value theory \cite{Arous-Voiculescu} there are useful homomorphisms between the semigroups on $[0, 1]$, but not such a perfect fit as an isomorphism.   

Section 4 uses the results of the preceding section to determine the Boolean max-stable laws and their domain of attraction. In essence, in view of the isomorphism of semigroup laws on $[0, 1]$ found in section 3, this amounts to transfers of classical results to the Boolean setting, with some restrictions. The main restriction is that the laws should be supported on $[0, \infty)$ and this singles out Fr\'echet laws in the classical framework. When transferred to Boolean probability these become Dagum distributions, which are also called log-logistic distributions. Thus, also in the Boolean framework the max-stable laws have also a history in classical probability. 

Concluding, we would like to mention some of the developments around the free extreme values of \cite{Arous-Voiculescu} (see \cite{Georges-Duvillard}, \cite{Grela-Nowak}, \cite{Hazra-Maulik}) which suggest similar questions for Boolean extremes. 

\section{Preliminary Considerations}

The noncommutative probability framework we use, that of hermitian operators on a Hilbert space (actually their spectral scales), is equivalent to using von Neumann algebras, but since this is a quite short paper, this seemed the best for our purpose. 

Thus, noncommutative random variables will be hermitian operators $T= T^*$ on some Hilbert space $\mathcal{H}$, endowed with a state vector $\xi \in \mathcal{H}$, $\norm{\xi} =1$, which gives the expectation functional  $\phi(T) = \langle T\xi, \xi\rangle$. We can encode the same data as a projection-valued process $\mathbb{R}  \ni t \mapsto E(T; (-\infty, t])$, which to each $t$ associates the spectral projection of $T$ for $(-\infty, t]$. Here, $E(T; (-\infty, t])$ is an increasing function of $t$, which is right semicontinuous with respect to the strong operator topology. That $T$ is bounded means that there are $t_0 < t_1$ so that $E(T; (-\infty, t_0]) =0$ and $E(T; (-\infty, t_1]) = I$. This works also for unbounded selfadjoint variables which will correspond to scales of spectral projections $[-\infty, \infty] \ni t \mapsto E(T; (\-\infty, t])$, which are increasing, right continuous and $E(T; (-\infty, t])$ is equal to 0 if $t=-\infty$ and equal to $I$ if $t= \infty$. There is an additional condition at $+\infty$, $\lim_{t\to +\infty} E(T; (-\infty, t]) = I$, that is, continuity at +$\infty$. 

A noncommutative random variable $a$ that is given by such a  scale has distribution function $$F_a(t) = \langle E(a; (-\infty, t])\xi, \xi \rangle \hspace{.1cm}, \hspace{.1cm} t\in \mathbb{R}$$
and its distribution is the probability measure $\mu_a$ on $\mathbb{R}$ with this distribution function, that is, $\mu_a((-\infty, t]) = F_a(t)$. The spectral max $a\vee b$ of $a$ and $b$ is given by the scale $E(a\vee b; (-\infty, t]) = E(a; (-\infty, t])\wedge E(b; (-\infty, t])$ where for two hermitian projections $P, Q$ by $P\wedge Q$ we denote the orthogonal projection onto the intersection of their ranges 
$$(P\wedge Q) (\mathcal{H}) = P(\mathcal{H})\wedge Q(\mathcal{H}). $$ 
Boolean independence of two families of bounded noncommutative random variables $(X_{\iota})_{\iota \in I}, (Y_j)_{j \in J}$ on some $(\mathcal{H}, \xi)$ can be described as follows. There exists $(X'_\iota)_{\iota \in I}$ on $(\mathcal{H}', \xi')$, and $(Y_j'')_{j\in J}$ on $(\mathcal{H}'', \xi'')$, with the same noncommutative moments with respect to $\xi'$ and $\xi''$ respectively, such that if we consider $(\mathcal{H}''', \xi''')$, where $\mathcal{H}'''= (\mathcal{H}'\ominus \mathbb{C}\xi') \oplus (\mathcal{H}''\ominus \mathbb{C}\xi'') \oplus \mathbb{C}\xi'''$, and identifying isometries $$V': \mathcal{H}'\to \mathcal{H}''',  V'': \mathcal{H}''\to \mathcal{H}''',$$ so that $V'\restriction (\mathcal{H}' \ominus \mathbb{C}\xi')$ and $V''\restriction (\mathcal{H}''\ominus \mathbb{C}\xi'')$ are the identity in the respective space, while $V'\xi'= V''\xi''= \xi'''$. Then, the joint distribution of $(X_\iota, Y_j, \iota\in I, j \in J)$ on $(\mathcal{H}, \xi)$ is the same as the joint distribution of $(\tilde{X_\iota}, \tilde{Y_j}, \iota \in I, j \in J )$ on $(\mathcal{H}''', \xi''')$, where $\tilde{X_\iota} = V'X_\iota' V'^*$ and $\tilde{Y_j} = V''X_j'' V''^*$. One feature of this construction is that if we construct Boolean independent algebras on $(\mathcal{H}''', \xi''')$ from $\mathcal{A}$ on $(\mathcal{H}', \xi')$ and $\mathcal{B}$ on $(\mathcal{H}'', \xi'')$, then, $V'\mathcal{A}V'^*$ and $V''\mathcal{B}V''^*$ do not contain the identity operator on $\mathcal{H}'''$, even if $\mathcal{A} \ni I_{\mathcal{H}'}$, $\mathcal{B} \ni I_{\mathcal{H}''}$. The same works for unbounded noncommutative random variables by using their spectral scales instead of the operators, but it must be noted that spectral scales from $\mathcal{H}', \mathcal{H}''$ to $\mathcal{H}'''$ do not transform by simply putting the $\sim$ on the projections. Actually we have the following:
\begin{itemize}
\item If $t < 0$ then 
$$E(\tilde{X'}, (-\infty, t]) = V'E(X'; (-\infty, t]) V'^*.$$
\item While if $t\geq 0$ then 
$$E(\tilde{X'}; (-\infty, t]) = V' E(X'; (-\infty, t])V'^*+P_{\mathcal{H}''\ominus \mathbb{C}\xi'''}.$$
\end{itemize} 
Similar formulae hold for the $\tilde{Y''}$. Clearly, if $X' \geq 0$ then the formulae are uniform. Also if $P', P''$ are projections on $\mathcal{H}', \mathcal{H}''$ then $\tilde{P'}\wedge \tilde{P''} \leq P_{\mathbb{C}\xi}$ so $\tilde{P'}\wedge \tilde{P''}$ is either 0 or $P_{\mathbb{C}\xi}$ and $\langle \tilde{P'} \wedge \tilde{P''}\xi, \xi \rangle \in \{0, 1\}$. On the other hand $(\tilde{P'}+P_{\mathcal{H}'\ominus \mathbb{C}\xi'}) \wedge (\tilde{P''}+P_{\mathcal{H}''\ominus \mathbb{C}\xi''})$ is no longer such a trivial quantity. 

Also observe that the joint distribution of the $(\tilde{X_{\iota}'}, \tilde{Y_j''}, \iota\in I, j\in J)$ is the same for all choices of $X'_\iota$, $Y_j''$ replicating our initial joint distribution. 

The conclusion of this discussion is that if we have two noncommutative random variables, replacing them with Boolean independent $\tilde{X}, \tilde{Y}$ ``copies in distribution" on another Hilbert space, the spectral max works well for nonnegative random variables $X\geq 0, Y \geq 0$ when we can choose $\tilde{X}\geq 0, \tilde{Y}\geq 0$. 

In the following, we will use the Boolean additive convolution (\cite{Bercovici}, \cite{Speicher-Woroudi}). If $X, Y$ are bounded noncommutative random variables which are Boolean independent, then the distribution of $X+Y$ can be computed out of the distribution of $X$ and $Y$ by Boolean convolution. The formula for computing this is as follows. If $\mu$ is a probability measure on $\mathbb{R}$ with compact support, let $G_\mu(z) = \int\frac{d \mu(t)}{z-t}$  be its Cauchy transform, where $z\in \mathbb{C}\setminus \text{supp}(\mu)$, and let 
$$K_\mu(z) = -\frac{1}{G_\mu(z)}+z. $$
Then the equality
$$K_{\mu_{X+Y}} = K_{\mu_X}+K_{\mu_Y},$$
allows to recover $G_{\mu_{X+Y}}$ from which $\mu_{X+Y}$ is obtained, by a standard procedure, using boundary values of the imaginary part from the upper half plane.

\section{The spectral max of Boolean independent projections and Boolean max convolution}

In view of the discussion in the preceding section, if $P$ is a projection, its distribution function is given by
$$F_P(t) = \begin{cases}0 & \text{if } t< 0 \\ p & \text{if } 0\leq t < 1 \\ 1 & \text{if } 1\leq t \end{cases}$$
where $p = \phi(I-P) = 1-\langle P\xi, \xi \rangle= \langle (I-P)\xi, \xi\rangle$. To find the formulae for Boolean max convolution we need to find these formulae for projections and this in turn can be done using additive Boolean convolution. 

\begin{lemma}
\label{lemmaphisup}
Let $P, Q$ be Boolean independent projections on $(\mathcal{H}, \xi)$ so that $\phi(P) = 1-p$, $\phi(Q) = 1-q$. Then $\phi (P\vee Q) = 1-r$ where 
$$r^{-1}-1 =(p^{-1}-1)+(q^{-1}-1).$$
\end{lemma}

\begin{proof}
We have $E(P\wedge Q; \{0\}) = E(P+Q; \{0\})$ and $\phi(E(P+Q); \{0\})$ will be computed using Boolean additive convolution. We have 
$$G_P(z)= \frac{p}{z}+\frac{1-p}{z-1} = \frac{z-p}{z(z-1)}$$
$$K_P(z) = -\frac{z(z-1)}{z-p}+z = \frac{z(1-p)}{z-p}$$
and similar formulae for $Q$. 
$$K_{P+Q}(z) = z \left( \frac{1-p}{z-p}+\frac{1-q}{z-q} \right).$$
Then we have 
$$G_{P+Q}(z) = \frac{1}{z-K_{P+Q}(z)} = \frac{1}{z}\cdot \frac{1}{1-\frac{1-p}{z-p}-\frac{1-q}{z-q}}.$$
Hence 
$$r= \lim_{z\to 0} z G_{P+Q}(z) = \lim_{z\to 0} \frac{1}{1-\frac{1-p}{z-p}-\frac{1-q}{z-q}} = \frac{1}{p^{-1}+q^{-1}-1}.$$
This implies 
$$r^{-1}-1 = (p^{-1}-1)+(q^{-1}-1).$$
\end{proof}
The preceding lemma defines a semigroup structure on $[0, 1]$, the operation of which we shall denote by $\minbool$, the symbol which should correspond to Boolean min convolution, which we will  consider only as an operation on $[0,1]$. So if $x, y \in [0, 1]$ then $(x\minbool y)^{-1}-1 = (x^{-1}-1)+(y^{-1}-1)$. Since, when dealing with classical independence the semigroup structure at which one arrives is multiplication $\cdot$ on $[0, 1]$, it is clearly important to compare the two. As the next lemma shows the result of the comparison is quite pleasing. 
\begin{lemma}
\label{lemmachiiso}
The semigroups $([0, 1], \minbool)$ and $([0, 1], \cdot)$ are isomorphic. The map $\chi: [0, 1] \to [0, 1]$ given by $$\chi(x) = \exp(1-x^{-1})$$ is an isomorphism which is also an order preserving homeomorphism. The inverse isomorphism, which is also order-preserving is given by $$\chi^{-1}(y) =(1-\log(y))^{-1}.$$
\end{lemma}

The proof of the previous lemma is straightforward and will be omitted.

From Lemma \ref{lemmaphisup} and the discussion in Section 2, the following definition is natural.

\begin{definition}
If $F_1, F_2$ are distribution functions on $[0, \infty)$ then their Boolean max-convolution is defined by 
$$(F_1 \maxbool F_2)(t) = F_1(t) \minbool F_2(t).$$
\end{definition}

\begin{lemma}
\label{lemmamaxbooldist}
If $X\geq 0, Y \geq 0$ are Boolean independent noncommutative random variables on $(\mathcal{H}, \xi)$ then 
$$F_{X \vee Y} (t) = F_X \maxbool F_Y(t). $$
The noncommutative random variables may be unbounded and in such case they are identified with their spectral scales.  
\end{lemma}

\begin{proof}
Since $X\vee Y$ amounts to $E(X\vee Y; [0, t])= E(X; [0, t])\wedge E(Y; [0, t])$ where $t\geq 0$, $X\geq 0$, $Y\geq 0$ and $E(X; [0, t])$ , $E(Y; [0, t])$ the lemma follows immediately from the preceding results. 
\end{proof}

\section{The transfer from classical to Boolean of max convolutions}

Let $\Delta_+$ denote the set of distribution functions for positive random variables, that is, functions $F: [0, \infty)\to [0,1]$ which are increasing, right continuous and which satisfy $\lim_{t\uparrow \infty} F(t)=1$. On $\Delta_+$ we consider the classical max convolution, which is just the multiplication $(F,G)\mapsto FG$ and the Boolean max convolution $(F, G) \mapsto F\maxbool G$ which for each $t\in [0, \infty)$ is given by $(F\maxbool G)(t) = F(t)\minbool G(t)$. In view of Lemma \ref{lemmachiiso}, the map $X: \Delta_+ \to \Delta_+$, where $X(F)= \chi \circ F$, or in other words, $X(F) = \exp(1-F^{-1})$, is an isomorphism. Hence, it is a bijection, preserves pointwise convergences and for every $F$ and $G$ it satisfies
$$X(F)\cdot X(G) = X(F\maxbool G).$$
The inverse map $X^{-1}: \Delta_+\to \Delta_+$ is given by $X^{-1}(F) = (1-\log(F))^{-1}$. This map has the property that $X^{-1}(F\cdot G) = X^{-1}(F) \maxbool X^{-1}(G)$. 

In addition to these properties we also have that replacing $F(t)$ by $F(at)$ for some $a > 0$ corresponds to the same operation on $X(F)$ or $X^{-1}(F)$, that is 
$$X(F(a \cdot )) = (X(F))(a\cdot )\hspace{.1cm} \text{ and } \hspace{.1cm} X^{-1}(F(a \cdot)) = (X^{-1}(F)) (a \cdot).$$
In view of the fact that this are the operations which we use to define max stable laws, whether classical or Boolean, we can use this correspondence to transfer classical results to Boolean results, the main restriction being that the spectral max on the max-convolution when we deal with Boolean independence will be restricted to nonnegative random variables, or equivalently, we will be restricted to distribution functions in $\Delta_+$. 

The definition of max-stable laws for the Boolean setting will not include shifts in the limit process, because this would mean using the ``constant" noncommutative random variable $I$ on $(\mathcal{H},\xi)$ and we prefer to avoid the technical problems that arise from considering Boolean independent copies of it. The downside of this will be that we will have to keep in mind which shifted limit laws may arise. 

\begin{definition}
A distribution function $F\in \Delta_+$ is a Boolean max-stable distribution function if for some $G\in \Delta_+$ there are constants $a_n > 0$, $n\in \mathbb{N}$ so that 
$$(\underbrace{G\maxbool \cdots \maxbool G}_{n})(a_nt) \to F(t),$$
for all $t\geq 0$. 
\end{definition}

The facts related to the transfer that were previously discussed yield the following result.

\begin{theorem}
A distribution $F\in \Delta_+$ is Boolean max-stable if and only if $X(F) (t) = \exp (-\lambda t^{-\alpha})$ for some $\lambda > 0$ and $\alpha > 0$. 
\end{theorem}

\begin{proof}
Suppose $G^{\maxbool n} (a_n \cdot ) \to F(\cdot)$ and let $K = X(G)$ and $H= X(F)$. It follows that $K^n(a_n\cdot ) \to H(\cdot)$. Hence $H(\cdot)$ is classical max-stable and its support is in $[0, \infty)$. So $H(t) = \exp(-\lambda(t-\mu)^{-\alpha})$ if $x\geq \mu$ where $\mu \geq 0$, $\lambda > 0$, $\alpha > 0$ and $H\restriction_{[0, \mu]} = 0$ by the classical result of extreme value theory (see \cite{Resnick}, Prop. 0.3.). By the proof of Prop. 0.3. in \cite{Resnick} we must have $H^n(t) = H(n^{-\theta} t)$ for some $\theta \in \mathbb{R}$. This implies that $n(t-\mu)^{-\alpha} = (n^{-\theta} t -\mu)^{-\alpha}$, that is $n^{-\frac{1}{\alpha}}t-n^{-\frac{1}{\alpha}} \mu= n^{-\theta}t-\mu$ so that we must have $\mu=0$. 

On the other hand, if $H(t) = \exp (-\lambda t^{-\alpha})$, $H(t) = \lim_{n\to \infty} K^n(a_nt)$ for some $K$ and $a_n$'s, for instance $K=H$ and $a_n = n^{-\frac{1}{\alpha}}$. 
\end{proof}

Using $X^{-1}$ we get from this the Boolean max-stable laws explicitly. 

\begin{corollary}
\label{Dagum}
$F\in \Delta_+$ is a Boolean max-stable law if and only if $$F(t) = (1+\lambda t^{-\alpha})^{-1}$$ where $\lambda > 0$ and $\alpha > 0$. 
\end{corollary}

The distribution appearing in Corollary \ref{Dagum} is known as a Dagum distribution or log-logistic distribution. Among other applications, the Dagum distribution is used in models of wealth distribution. 

The transfer between classical and Boolean max-stable laws, which we fortunately posses, also gives results on domains of attraction. 

\begin{definition}
A distribution function $G\in \Delta_+$ is in the Boolean-max domain of attraction of $F\in \Delta_+$ if for some sequence of $a_n > 0$, $n\in \mathbb{N}$ we have $G^n(a_n \cdot) \to F(\cdot)$ as $n\to \infty$. 
\end{definition}

The classical correspondents of the Dagum laws (viewed as Boolean max-stable laws) are Fr\'echet laws and their classical domains of attraction are known without using shifts (see Prop. 1. 11. in \cite{Resnick}, a classical result of Gnedenko). The condition for a distribution function $K \in \Delta_+$, to be in the domain of attraction of a Fr\'echet distribution, is that $1-K$ has to be regularly varying of exponent $-\alpha$, that is 
$$\lim_{t\to +\infty} \frac{1-K(tx)}{1-K(t)} = x^{-\alpha},$$
for all $x > 0$. This gives immediately the following result. 

\begin{corollary}
$G\in \Delta_+$ is in the Boolean domain of attraction of $F(t) = (1+t^{-\alpha})^{-1}$ if and only if $1-\exp(1-G^{-1})$ is regularly varying of exponent $-\alpha$. 
\end{corollary}

The condition on $G$ in the corollary is early seen to be equivalent to $1-G$ being regularly varying of index $-\alpha$. Indeed, since
$$\lim_{t\to \infty} \left(1-(G(t))^{-1} \right) = \lim_{t\to \infty} \left( 1- (G(tx))^{-1} \right) =0 $$
and $\lim_{y\to 0} (\exp y-1) / y =1$, the condition on $G$ is equivalent to
$$\lim_{t\to \infty} \frac{1-(G(tx))^{-1}}{1-(G(t))^{-1}}= x^{-\alpha}.$$
Thus, like shown in \cite{Arous-Voiculescu} for free max-convolution, also for Boolean max-convolution, a Bercovici-Pata type coincidence of the domains of attraction with the classical domains of attractions takes place. 

\begin{corollary}
A distribution function $F$ is in the max-Boolean domain of attraction of the Dagum distribution function $(1+t^{-\alpha})^{-1}$ if and only if $1-F$ is regulary varying of index $-\alpha$.  
\end{corollary}

\end{document}